\documentclass[11pt]{amsart}
\newtheorem{theorem}{Theorem}[section]
\newtheorem{corollary}[theorem]{Corollary}
\newtheorem{lemma}[theorem]{Lemma}
\newtheorem{proposition}[theorem]{Proposition}

\theoremstyle{definition}

\newtheorem{remark}[theorem]{Remark}
\numberwithin{equation}{section}

\newcommand{\real}{\mathbb R}

\def\natu{\mathbb N}

\begin{document}

\title[Boundedness of the extremal solutions in dimension $4$] {Boundedness of the extremal solutions in dimension $4$}
\author{Salvador Villegas}
\thanks{The author has been supported by the MEC Spanish grants
MTM2008-00988 and MTM2009-10878}
\address{Departamento de An\'{a}lisis
Matem\'{a}tico, Universidad de Granada, 18071 Granada, Spain.}
\email{svillega@ugr.es}

\begin{abstract}
In this paper we establish the boundedness of the extremal solution $u^\ast$ in dimension $N=4$ of the semilinear elliptic
equation $-\Delta u=\lambda f(u)$, in a general smooth bounded domain $\Omega \subset \mathbb{R}^N$, with Dirichlet
data $u|_{\partial \Omega}=0$, where $f$ is a $C^1$ positive,
nondecreasing and convex function in $[0,\infty)$ such that
$f(s)/s\rightarrow\infty$ as $s\rightarrow\infty$.

In addition, we prove that, for $N\geq 5$, the extremal solution $u^*\in W^{2,\frac{N}{N-2}}$. This gives $u^\ast\in L^\frac{N}{N-4}$, if $N\geq 5$ and $u^*\in H_0^1$, if $N=6$.
\end{abstract}

\maketitle

\section{Introduction and main results}

In this paper, we consider the following semilinear elliptic equation, which has been extensively studied:
$$
\left\{
\begin{array}{ll}
-\Delta u=\lambda f(u)\ \ \ \ \ \ \  & \mbox{ in } \Omega \, ,\\
u\geq 0 & \mbox{ in } \Omega \, ,\\
u=0  & \mbox{ on } \partial\Omega \, ,\\
\end{array}
\right. \eqno{(P_\lambda)}
$$
\

\noindent where $\Omega\subset\real^N$ is a smooth bounded domain,
$N\geq 1$, $\lambda\geq 0$ is a real parameter and the
nonlinearity $f:[0,\infty)\rightarrow \real$ satisfies

\begin{equation}\label{convexa}
f \mbox{ is } C^1, \mbox{ nondecreasing and convex, }f(0)>0,\mbox{
and }\lim_{u\to +\infty}\frac{f(u)}{u}=+\infty.
\end{equation}

It is well known that there exists a finite positive extremal
parameter $\lambda^\ast$ such that ($P_\lambda$) has a minimal
classical solution $u_\lambda\in C^2(\overline{\Omega})$ if $0\leq
\lambda <\lambda^\ast$, while no solution exists, even in the weak
sense, for $\lambda>\lambda^\ast$. The set $\{u_\lambda:\, 0\leq
\lambda < \lambda^\ast\}$ forms a branch of classical solutions
increasing in $\lambda$. Its increasing pointwise limit
$u^\ast(x):=\lim_{\lambda\uparrow\lambda^\ast}u_\lambda(x)$ is a
weak solution of ($P_\lambda$) for $\lambda=\lambda^\ast$, which
is called the extremal solution of ($P_\lambda$) (see
\cite{Bre,BV,Dup}). In fact, if $f$ satisfies all the hypotheses of
(\ref{convexa}) except the convexity, then all the results we have mentioned
remain true, except the continuity of the family of minimal solutions
$\{ u_\lambda \}$ as a function of $\lambda$ (see \cite[Proposition 5.1]{cc}).

The regularity and properties of the extremal solutions depend
strongly on the dimension $N$, domain $\Omega$ and nonlinearity
$f$. When $f(u)=e^u$, it is known that $u^\ast\in L^\infty
(\Omega)$ if $N<10$ (for every $\Omega$) (see \cite{CrR,MP}),
while $u^\ast (x)=-2\log \vert x\vert$ and $\lambda^\ast=2(N-2)$
if $N\geq 10$ and $\Omega=B_1$ (see \cite{JL}). There is an
analogous result for $f(u)=(1+u)^p$ with $p>1$ (see \cite{BV}).
Brezis and V\'azquez \cite{BV} raised the question of determining
the boundedness of $u^\ast$, depending on the dimension $N$, for
general nonlinearities $f$ satisfying (\ref{convexa}). The first general results were due to Nedev \cite{Ne}, who proved that $u^\ast \in
L^\infty (\Omega)$ if $N\leq 3$, and $u^\ast \in
L^p (\Omega)$ for every $p<N/(N-4)$, if $N\geq 4$. The best known result was established by Cabr\'e \cite{cabre4}, who proved that $u^\ast \in
L^\infty (\Omega)$ if $N\leq 4$ and $\Omega$ is convex (no convexity on $f$ is imposed). If $N\geq 5$ and $\Omega$ is convex Cabr\'e and Sanch\'on \cite{casa} have obtained that $u^\ast \in
L^\frac{2N}{N-4} (\Omega)$ (again, no convexity on $f$ is imposed). On the other hand, Cabr\'e and Capella \cite{cc} have proved
that $u^\ast \in L^\infty (\Omega)$ if $N\leq 9$ and $\Omega=B_1$. Recently, Cabr\'e and Ros-Oton \cite{cros} have obtained that $u^\ast \in L^\infty (\Omega)$ if $N\leq 7$ and $\Omega$ is a convex domain of double revolution (see \cite{cros} for the definition).

Another interesting question is whether the
extremal solution lies in the energy class.  Nedev \cite{Ne,Ne2}
proved that $u^\ast \in H_0^1(\Omega)$ if $N\leq 5$ (for every
$\Omega$) or $\Omega$ is convex (for every $N\geq 1$). Brezis and V\'azquez \cite{BV} proved that a sufficient condition
to have $u^\ast \in H_0^1(\Omega)$ is that $\liminf_{u\to \infty}
u\, f'(u)/f(u)>1$ (for every $\Omega$ and $N\geq 1$).

In this paper we establish the boundedness of the extremal solution for general bounded smooth domains in dimension $4$, not necessarily convex. Contrary to the result of Cabr\'e, we need to impose the convexity of $f$. In higher dimensions, we improve the results of Nedev \cite{Ne,Ne2} and it is obtained that $u^\ast \in L^\frac{N}{N-4}(\Omega)$, if $N\geq 5$ and $u^\ast \in H_0^1(\Omega)$, if $N=6$.

\begin{theorem}\label{N=4}

Let $f$ be a function satisfying (\ref{convexa}) and $\Omega\subset \mathbb{R}^4$ be a smooth bounded domain. Let $u^\ast$ be the extremal
solution of ($P_\lambda$). Then $u^\ast\in L^\infty (\Omega)$.

\end{theorem}

\begin{theorem}\label{N>4}

Let $f$ be a function satisfying (\ref{convexa}) and $\Omega\subset \mathbb{R}^N$ be a smooth bounded domain. Let $u^\ast$ be the extremal
solution of ($P_\lambda$). Then, for $N\geq 5$, $u^\ast\in W^{2,\frac{N}{N-2}}(\Omega)$ and $f(u^\ast)\in L^\frac{N}{N-2}(\Omega)$ . In particular,

\begin{enumerate}

\item[i)]If $N\geq 5$, then $\displaystyle{u^\ast\in L^\frac{N}{N-4}(\Omega)}$.
\item[ii)] If $N=6$, then $u^\ast \in  H_0^1(\Omega)$.

\end{enumerate}

\end{theorem}

The proofs of Theorems \ref{N=4} and \ref{N>4} use the semi-stability of of the minimal solutions $u_\lambda$ ($0<\lambda<\lambda^\ast$).

Recall that a classical solution $u$ of

\begin{equation}\label{general}
\left\{
\begin{array}{ll}
-\Delta u=g(u)\ \ \ \ \ \ \  & \mbox{ in } \Omega \, ,\\
u=0  & \mbox{ on } \partial\Omega \, ,\\
\end{array}
\right.
\end{equation}

\noindent where $N\geq 1$, $g\in C^1(\mathbb{R})$ and $\Omega\subset \mathbb{R}^N$ is a smooth bounded domain, is semistable if

$$\int_\Omega \left( \vert \nabla \xi\vert^2-g'(u)\xi^2\right) \,
dx\geq 0\, ,$$ \noindent for every $\xi\in C^\infty (\Omega)$ with compact
support in $\Omega$.

Note that this expression is the second variation of energy at $u$. The semistability of a solution $u$ is equivalent to the nonnegativity of $\lambda_1\left(-\Delta-g'(u);\Omega\right)$, the first Dirichlet eigenvalue of the linearized operator $-\Delta-g'(u)$ at $u$ in $\Omega$.

To prove our main results we will use the following lemma, which follows easily from a result of Nedev \cite{Ne}.

\begin{lemma}\label{est}

Let $N\geq 1$, $f$ be a function satisfying (\ref{convexa}) and $\Omega\subset \mathbb{R}^N$ be a smooth bounded domain. Then there exists a positive constant $M=M(f,\Omega)$, depending on $f$ and $\Omega$, but not on $\lambda \in (0,\lambda^\ast)$, such that

$$\int_{u_\lambda>1} \frac{f(u_\lambda)^2}{u_\lambda}\leq M\, , \ \ \ \forall \lambda\in (0,\lambda^\ast).$$

\end{lemma}

\begin{proof}

Using the semistability of the minimal solutions $u_\lambda$, Nedev (see the proof of Theorem 1 in \cite{Ne}) obtained that

\begin{equation}\label{nedev}\int_\Omega \frac{(f(u_\lambda)-f(0))^2}{u_\lambda}\leq M_1 \, , \ \ \ \forall \lambda\in (0,\lambda^\ast),\end{equation}

\noindent where $M_1$ is a constant independent of $\lambda$. On the other hand, since $\lim_{s\to +\infty}f(s)=+\infty$, then $\lim_{s\to +\infty}\left(2(f(s)-f(0))^2-f(s)^2\right)=+\infty$. Thus $2(f(s)-f(0))^2-f(s)^2\geq -M_2$, for every $s\geq 0$, where $M_2$ is a constant depending only on $f$. Applying this and (\ref{nedev}), we conclude that

$$\int_{u_\lambda>1} \frac{f(u_\lambda)^2}{u_\lambda}\leq\int_{u_\lambda>1} \frac{M_2+2\left(f(u_\lambda)-f(0)\right)^2}{u_\lambda}\leq M_2\vert \Omega \vert+2M_1,$$

\noindent and the lemma follows.
\end{proof}

The paper is organized as follows. Section \ref{N==4} deals with dimension $N$=4 and we prove Theorem \ref{N=4}. In Section \ref{N>>4}, the estimates of Theorem \ref{N>4} are proved. Finally, in Section \ref{sobolevv} we obtained some new $W^{1,q}$ and $W^{2,q}$ estimates of the extremal solution $u^\ast$.

\section{The case $N=4$}\label{N==4}

The following theorem is due to Cabr\'e, and it is the main estimate used in the proof of the results of \cite{cabre4}. We will use it, in order to obtain Lemma \ref{C1}.

\begin{theorem}(\cite{cabre4}).\label{desigualdades}
Let $g$ be any $C^\infty$-function and $\Omega\subset \mathbb{R}^N$ any smooth bounded domain. Assume that $2\leq N\leq 4$.

Let $u\in C^1_0(\overline{\Omega})$, with $u>0$ in $\Omega$, a classical semistable solution of (\ref{general}). Then, for every $t>0$,

$$
\Vert u\Vert_{L^\infty (\Omega)}\leq t+\frac{K}{t}\vert\Omega\vert^\frac{4-N}{2N}\left(\int_{u<t}\vert\nabla u\vert^4\right)^{1/2},
$$

\noindent where $K$ is a universal constant (in particular, independent of $g$, $\Omega$ and $u$).

\end{theorem}

\begin{lemma}\label{C1}
Let $g$ be any $C^1$-function satisfying $g(0)>0$ and $\Omega\subset \mathbb{R}^N$ any smooth bounded domain. Assume that $2\leq N\leq 4$.

Let $u\in C^1_0(\overline{\Omega})$, with $u>0$ in $\Omega$, a classical minimal positive solution of (\ref{general}) (i.e., $u$ is the only solution of (\ref{general}) in the set $\left\{ w\in C^1_0(\overline{\Omega}):\, 0\leq w\leq u\right\}$). Then, for every $t>0$,

$$
\Vert u\Vert_{L^\infty (\Omega)}\leq t+\frac{K}{t}\vert\Omega\vert^\frac{4-N}{2N}\left(\int_{u<t}\vert\nabla u\vert^4\right)^{1/2},
$$

\noindent where $K$ is a universal constant (in particular, independent of $g$, $\Omega$ and $u$). In fact, we can take  the same constant $K$ of Theorem \ref{desigualdades}.

\end{lemma}

\begin{proof} Let $L=\Vert u \Vert_{L^\infty(\Omega)}$. Take a sequence of polynomials $p_n$ such that $p_n(x)<g(x)$ for every $x\in [0,L]$ and $p_n\to g$ in $L^\infty (0,L)$ as $n\to\infty$. (Take for instance $p_n$ such that $g-2/n\leq p_n\leq g-1/n$ in $[0,L]$). Hence $u$ is a strict supersolution of the problem

$$
\left\{
\begin{array}{ll}
-\Delta w=p_n(w)\ \ \ \ \ \ \  & \mbox{ in } \Omega \, ,\\
w=0  & \mbox{ on } \partial\Omega \, ,\\
\end{array}
\right. \eqno{(P_n)}
$$
\

On the other hand, since $g(0)>0$ and $p_n\to g$ in $L^\infty (0,L)$ as $n\to\infty$, we have that, up to a subsequence, $p_n(0)>0$ for every $n\in \natu$. This is equivalent to the fact that the trivial function $0$ is a strict subsolution of the problem $(P_n)$. Then, the energy functional for this equation is well defined in the closed convex set of $H_0^1(\Omega)$ functions $w$ satisfying $0\leq w\leq u$, and it admits an absolute minimizer $u_n$ in this convex set. It is well known that $u_n$ is a classical semistable solution of $(P_n)$ (see \cite[Rem. 1.11]{cc} for more details). Therefore, by Theorem \ref{desigualdades}

\begin{equation}\label{generaln}
\begin{array}{ll}
\displaystyle{\Vert u_n\Vert_{L^\infty (\Omega)}\leq t+\frac{K}{t}\vert\Omega\vert^\frac{4-N}{2N}\left(\int_{u_n<t}\vert\nabla u_n\vert^4\right)^{1/2}, \ \forall t>0.}
\end{array}
\end{equation}

Since $\Vert u_n \Vert_{L^\infty(\Omega)}\leq L$, then $\Vert p_n(u_n) \Vert_{L^\infty(\Omega)}\leq \Vert p_n \Vert_{L^\infty(0,L)}\leq C'$, for some constant $C'$. Thus, by elliptic regularity (see \cite{ADN}), $\Vert u_n \Vert_{W^{2,p}(\Omega)}$ is bounded for every $1<p<\infty$. Choosing $p>N$, we can suppose, up to a subsequence, that $u_n\rightharpoonup u_0$ in $W^{2,p}(\Omega)$ and $u_n\rightarrow u_0$ in $C^1_0(\overline{\Omega})$ for some function $u_0\in W^{2,p}(\Omega)$. On the other hand

$$\Vert p_n(u_n)-g(u)\Vert_{L^\infty(\Omega)}\leq \Vert p_n(u_n)-g(u_n)\Vert_{L^\infty(\Omega)}+\Vert g(u_n)-g(u)\Vert_{L^\infty(\Omega)}$$

$$\leq \Vert p_n -g\Vert_{L^\infty(0,L)}+\Vert u_n-u\Vert_{L^\infty(\Omega)}\Vert g'\Vert_{L^\infty(0,L)}\rightarrow 0, \mbox{ as } n\rightarrow \infty.$$

Thus $p_n(u_n)\rightarrow g(u)$ in $L^\infty (\Omega)$ and it follows easily that $u_0$ is a classical solution of (\ref{general}). Since $0\leq u_0\leq u$ and $u$ is a classical minimal positive solution of (\ref{general}), we deduce that $u_0=u$. Hence $u_n\rightarrow u$ in $C_0^1(\overline{\Omega})$.

We claim that $\chi_{\{u_n\leq t\} }(x)\to \chi_{\{u\leq t\} }(x)$ for every $x\in \Omega$ and $t>0$. Here $\chi_A$ denotes the characteristic function of the set $A$. Indeed, if $u(x)\leq t$ then $u_n(x)\leq u(x)\leq t$ and consequently $\chi_{\{u_n\leq t\}}(x)=\chi_{\{u\leq t\}}(x)=1$. If $u(x)>t$ then, by the  $L^\infty(\Omega)$ convergence of $u_n$ to $u$, we have that $u_n(x)>t$ for large $n$ and it follows that $\chi_{\{u_n\leq t\}}(x)\to 0=\chi_{\{u\leq t\}}(x)$.

Applying this, the convergence of $u_n$ to $u$ in $C_0^1(\overline{\Omega})$, Sard's Theorem and the Lebesgue's dominated convergence Theorem we conclude, for every $t>0$, that

$$\int_{u_n<t} \vert \nabla u_n\vert^4=\int_{u_n\leq t} \vert \nabla u_n\vert^4=\int_\Omega \vert\nabla u_n\vert^4 \chi_{\{u_n\leq t\}}\rightarrow \int_\Omega \vert\nabla u\vert^4 \chi_{\{u\leq t\}}$$

$$=\int_{u\leq t} \vert \nabla u\vert^4=\int_{u<t} \vert \nabla u\vert^4.$$

Thus, taking limit as $n$ tends to $\infty$ in (\ref{generaln}), the proof is complete.
\end{proof}

\begin{proposition}\label{des}

Let $g$ be a function and $\Omega\subset \mathbb{R}^4$ any smooth bounded domain. Let $u\in C^1_0(\overline{\Omega})$, with $u>0$ in $\Omega$, a classical solution of (\ref{general}). Suppose that one of the following holds:

\begin{enumerate}
\item[(i)] $g\in C^\infty$ and $u$ is a semistable solution.
\item[(ii)] $g\in C^1$, $g(0)>0$ and $u$ is a minimal positive solution (i.e., $u$ is the only solution of (\ref{general}) in the set $\left\{ w\in C^1_0(\overline{\Omega}):\, 0\leq w\leq u\right\}$).
\end{enumerate}

Then, there exists a universal constant C (in particular, independent of $g$, $\Omega$, and $u$) such that

$$\Vert u\Vert_{L^\infty (\Omega)}\leq C \Vert \nabla u\Vert_{L^4 (\Omega)}.$$

\end{proposition}

\begin{proof} Applying Theorem \ref{desigualdades} and Lemma \ref{C1} with $N=4$, we can assert that

$$
\Vert u\Vert_{L^\infty (\Omega)}\leq t+\frac{K}{t}\left(\int_{u<t}\vert\nabla u\vert^4 \right)^{1/2}\leq t+\frac{K}{t}\left(\int_\Omega\vert\nabla u\vert^4 \right)^{1/2}, \, \forall t>0.
$$

Taking $\displaystyle{t=\left(\int_\Omega\vert\nabla u\vert^4 \right)^{1/4}}$ in this expression we obtain

$$
\Vert u\Vert_{L^\infty (\Omega)}\leq\left(\int_\Omega\vert\nabla u\vert^4 \right)^{1/4}+\frac{K}{\displaystyle{\left(\int_\Omega\vert\nabla u\vert^4 \right)^{1/4}}}\left(\int_\Omega\vert\nabla u\vert^4 \right)^{1/2}$$

$$=\left(1+K\right)\Vert \nabla u\Vert_{L^4 (\Omega)},$$

\

\noindent and the lemma follows with $C=1+K$.
\end{proof}

\begin{remark}\label{C} From classical embedding results of Sobolev spaces, it is well-known that, for a smooth bounded domain $\Omega\subset \mathbb{R}^4$,  we have the continuous inclusions $W^{1,4}\subset L^p$, for every $1\leq p<\infty$, and $W^{1,4+\varepsilon}\subset L^\infty$, for every $\varepsilon >0$. On the other hand it is also well-known that $W^{1,4}\not\subset L^\infty$, which is equivalent to the unboundness of the quotients $\Vert u\Vert_{L^\infty (\Omega)}/\Vert \nabla u\Vert_{L^4 (\Omega)}$, $u\in C^1_0(\overline{\Omega})$, with $u>0$ in $\Omega$. The previous proposition asserts that, under some stability hypothesis on $u$, these quotients are bounded.
\end{remark}

\noindent {\bf Proof of Theorem \ref{N=4}.} It is well-known that, for every smooth domain $\Omega\subset\mathbb{R}^4$, we have the continuous inclusion $W^{2,2}(\Omega)\subset W^{1,4}(\Omega)$. Thus, there exists a constante $C_1=C_1(\Omega)$, depending only on $\Omega$, such that

\begin{equation}\label{sobolev}\Vert u\Vert_{W^{1,4}(\Omega)}\leq C_1\Vert u\Vert_{W^{2,2}(\Omega)}\, , \ \ \ \ \ \forall u\in W^{2,2}(\Omega).\end{equation}

On the other hand, by elliptic regularity (see \cite{ADN}), there exists a constant $C_2=C_2(\Omega)$, depending only on $\Omega$, such that

\begin{equation}\label{regularity}\Vert u_\lambda\Vert_{W^{2,2}(\Omega)}\leq C_2\Vert \lambda f(u_\lambda) \Vert_{L^2(\Omega)}.\end{equation}

Applying (\ref{sobolev}), (\ref{regularity}), Lemma \ref{est} and Proposition \ref{des} (part (ii), with $g=\lambda f$), we deduce, for every $\lambda\in (0,\lambda^\ast)$, that

$$\Vert u_\lambda\Vert_{L^\infty}\leq C \Vert \nabla u_\lambda\Vert_{L^4}\leq C\Vert u_\lambda\Vert_{W^{1,4}}\leq  C C_1 \Vert u_\lambda\Vert_{W^{2,2}}\leq C C_1 C_2\Vert \lambda f(u_\lambda) \Vert_{L^2}$$

$$\leq C C_1 C_2 \lambda^\ast \left( \int_{u_\lambda\leq 1} f(u_\lambda)^2+\int_{u_\lambda >1} f(u_\lambda)^2\right)^{1/2}$$

$$\leq C C_1 C_2 \lambda^\ast \left( f(1)^2\vert \Omega \vert+\int_{u_\lambda>1} \frac{f(u_\lambda)^2}{u_\lambda}u_\lambda\right)^{1/2}$$

$$\leq C C_1 C_2 \lambda^\ast \left( f(1)^2\vert \Omega \vert+M \Vert u_\lambda\Vert_{L^\infty} \right)^{1/2}.$$

\

Therefore $\Vert u_\lambda\Vert_{L^\infty}^2\leq A +B\Vert u_\lambda\Vert_{L^\infty}$, for certain constant $A,B$ depending on $f$ and $\Omega$, but not on $\lambda\in (0,\lambda^\ast)$. We conclude that $\Vert u_\lambda\Vert_{L^\infty}$ is uniformly bounded in $\lambda\in(0,\lambda^\ast)$, and finally, taking limit $\lambda \to \lambda^\ast$, that $u^\ast \in L^\infty (\Omega)$. \qed

\section{The case $N\geq 5$}\label{N>>4}

\noindent {\bf Proof of Theorem \ref{N>4}.} Let $N\geq 5$. It is well-known that, for every smooth domain $\Omega\subset\mathbb{R}^N$ and exponent $1<p<N/2$, we have the continuous inclusion $W^{2,p}(\Omega)\subset L^\frac{Np}{N-2p}(\Omega)$. Thus, taking $p=N/(N-2)$, there exists a constante $C_3=C_3(\Omega)$, depending only on $\Omega$, such that

\begin{equation}\label{sobolevN}\Vert u\Vert_{L^\frac{N}{N-4}(\Omega)}\leq C_3\Vert u\Vert_{W^{2,\frac{N}{N-2}}(\Omega)}\, , \ \ \ \ \ \forall u\in W^{2,\frac{N}{N-2}}(\Omega).\end{equation}

On the other hand, by elliptic regularity (see \cite{ADN}), there exists a constant $C_4=C_4(\Omega)$, depending only on $\Omega$, such that

\begin{equation}\label{regularityN}\Vert u_\lambda\Vert_{W^{2,\frac{N}{N-2}}(\Omega)}\leq C_4\Vert \lambda f(u_\lambda) \Vert_{L^\frac{N}{N-2}(\Omega)}.\end{equation}

Applying (\ref{sobolevN}), (\ref{regularityN}), Lemma \ref{est} and H{\"o}lder inequality, we deduce, for every $\lambda\in (0,\lambda^\ast)$, that

$$\Vert u_\lambda\Vert_{W^{2,\frac{N}{N-2}}(\Omega)}\leq C_4 \Vert \lambda f(u_\lambda) \Vert_{L^\frac{N}{N-2}(\Omega)}$$

$$\leq C_4 \lambda^\ast \left( \int_{u_\lambda\leq 1} f(u_\lambda)^\frac{N}{N-2}+\int_{u_\lambda >1} f(u_\lambda)^\frac{N}{N-2}\right)^\frac{N-2}{N}$$

$$\leq C_4 \lambda^\ast \left( f(1)^\frac{N}{N-2}\vert \Omega \vert+\int_{u_\lambda>1} \left(\frac{f(u_\lambda)^2}{u_\lambda}\right)^\frac{N}{2(N-2)}u_\lambda^\frac{N}{2(N-2)}\right)^\frac{N-2}{N}$$

$\leq C_4 \lambda^\ast \left( f(1)^\frac{N}{N-2}\vert \Omega \vert+\left\Vert\left(\frac{f(u_\lambda)^2}{u_\lambda}\right)^\frac{N}{2(N-2)}\right\Vert_{L^\frac{2(N-2)}{N}\left(\left\{u_\lambda>1\right\}\right)} \right.$

$ \ \ \ \ \ \ \ \ \ \ \ \ \ \ \ \ \ \ \ \ \ \ \ \ \ \ \ \ \ \ \ \ \ \ \ \ \ \ \ \ \ \ \ \ \ \ \ \ \ \ \ \ \ \ \ \ \ \ \ \ \ \ \ \ \ \ \ \ \ \ \ \ \ \ \ \ \ \ \ \ \ \left. \times\left\Vert u_\lambda^\frac{N}{2(N-2)}\right\Vert_{L^\frac{2(N-2)}{N-4}\left(\left\{u_\lambda>1\right\}\right)}\right)^\frac{N-2}{N}$

$$\leq C_4 \lambda^\ast \left( f(1)^\frac{N}{N-2}\vert \Omega \vert+M^\frac{N}{2(N-2)}\Vert u_\lambda\Vert_{L^\frac{N}{N-4}(\Omega)}^\frac{N}{2(N-2)}\right)^\frac{N-2}{N}$$

$$\leq C_4 \lambda^\ast \left( f(1)^\frac{N}{N-2}\vert \Omega \vert+M^\frac{N}{2(N-2)}\left( C_3 \Vert u_\lambda\Vert_{W^{2,\frac{N}{N-2}}(\Omega)}\right)^\frac{N}{2(N-2)}\right)^\frac{N-2}{N}.$$

\

Therefore $\Vert u_\lambda\Vert_{W^{2,\frac{N}{N-2}}(\Omega)}^\frac{N}{N-2}\leq A +B\Vert u_\lambda\Vert_{W^{2,\frac{N}{N-2}}(\Omega)}^\frac{N}{2(N-2)}$, for certain constant $A,B$ depending on $f$ and $\Omega$, but not on $\lambda\in (0,\lambda^\ast)$. It follows that $\displaystyle{\Vert u_\lambda\Vert_{W^{2,\frac{N}{N-2}}(\Omega)}}$ is uniformly bounded in $\lambda\in(0,\lambda^\ast)$ and, taking into account the previous inequalities, $\displaystyle{\Vert f(u_\lambda)\Vert_{L^\frac{N}{N-2}}(\Omega)}$ is also uniformly bounded in $\lambda\in(0,\lambda^\ast)$. Therefore, taking limit $\lambda \to \lambda^\ast$, we deduce that $u^\ast \in W^{2,\frac{N}{N-2}}(\Omega)$ and $f(u^\ast) \in L^\frac{N}{N-2}(\Omega)$.

Finally, since $N\geq 5$, we have that $\displaystyle{W^{2,\frac{N}{N-2}}(\Omega)\subset W^{1,\frac{N}{N-3}}(\Omega)\subset L^\frac{N}{N-4}(\Omega)}$ and we conclude i) and ii). \qed

\section{Some new $W^{1,q}$ and $W^{2,q}$ estimates}\label{sobolevv}

\begin{proposition}\label{fijate}

Let $N\geq 5$, $f$ be a function satisfying (\ref{convexa}) and $\Omega\subset \mathbb{R}^N$ be a smooth bounded domain. Let $u^\ast$ be the extremal
solution of ($P_\lambda$). Suppose that $u^\ast \in L^p(\Omega)$ for some $p\in (1,\infty)$. Then $f(u^\ast)\in L^\frac{2p}{p+1}(\Omega)$ and $u^\ast \in W^{2,\frac{2p}{p+1}}(\Omega)\subset W^{1,\frac{2pN}{(p+1)N-2p}}(\Omega)$.

\end{proposition}

\begin{proof} Applying Lemma \ref{est} and H{\"o}lder inequality, we deduce, for every $\lambda\in (0,\lambda^\ast)$, that

$$\Vert f(u_\lambda) \Vert_{L^\frac{2p}{p+1}(\Omega)}^\frac{2p}{p+1}=\int_{u_\lambda\leq 1} f(u_\lambda)^\frac{2p}{p+1}+\int_{u_\lambda >1}\left(\frac{f(u_\lambda)^2}{u_\lambda}\right)^\frac{p}{p+1}u_\lambda^\frac{p}{p+1}$$

$$\leq f(1)^\frac{2p}{p+1}\vert \Omega \vert+\left\Vert\left(\frac{f(u_\lambda)^2}{u_\lambda}\right)^\frac{p}{p+1}\right\Vert_{L^\frac{p+1}{p}\left(\left\{u_\lambda>1\right\}\right)}\left\Vert u_\lambda^\frac{p}{p+1}\right\Vert_{L^{p+1}\left(\left\{u_\lambda>1\right\}\right)}$$

$$\leq f(1)^\frac{2p}{p+1}\vert \Omega \vert+M^\frac{p}{p+1}\Vert u_\lambda\Vert_{L^p(\Omega)}^\frac{p}{p+1}\leq f(1)^\frac{2p}{p+1}\vert \Omega \vert+M^\frac{p}{p+1}\Vert u^\ast\Vert_{L^p(\Omega)}^\frac{p}{p+1}.$$

Letting $\lambda\uparrow\lambda^\ast$ and using the monotone convergence Theorem, we deduce that $f(u^\ast)\in L^\frac{2p}{p+1}(\Omega)$.

On the other hand, by elliptic regularity (\cite{ADN}), there exists $C_5=C_5(p,\Omega)$, depending only on $p$ and $\Omega$, such that

$$\Vert u_\lambda\Vert_{W^{2,\frac{2p}{p+1}}(\Omega)}\leq C_5 \left\Vert \lambda f(u_\lambda)\right\Vert_{L^\frac{2p}{p+1}(\Omega)}\leq C_5 \lambda^\ast \left\Vert f(u^\ast)\right\Vert_{L^\frac{2p}{p+1}(\Omega)}, \forall \lambda\in (0,\lambda^\ast).$$

Hence $\Vert u_\lambda\Vert_{W^{2,\frac{2p}{p+1}}(\Omega)}$ is uniformly bounded in $\lambda\in (0,\lambda^\ast)$. We conclude that $u^\ast \in W^{2,\frac{2p}{p+1}}(\Omega)\subset W^{1,\frac{2pN}{(p+1)N-2p}}(\Omega)$.
\end{proof}

\begin{corollary}\label{mirapordonde}

Let $5\leq N\leq 7$, $f$ be a function satisfying (\ref{convexa}) and $\Omega\subset \mathbb{R}^N$ be a convex smooth bounded domain. Let $u^\ast$ be the extremal
solution of ($P_\lambda$). Then $u^\ast \in W_0^{1,\frac{4N}{3N-8}}(\Omega)$.

\end{corollary}

\begin{proof} As we have mentioned in the Introduction, by a result of Cabr\'e and Sanch\'on \cite{casa}, if $\Omega$ is convex then $u^\ast\in L^\frac{2N}{N-4}(\Omega)$. Applying Proposition \ref{fijate} with $p=2N/(N-4)$, we obtain $u^\ast \in W^{2,\frac{4N}{3N-4}}(\Omega)\subset W^{1,\frac{4N}{3N-8}}(\Omega)$ and the corollary follows.
\end{proof}

\begin{remark} In the proof of the previous corollary we have not used that $N\leq 7$. It is immediate that $4N/(3N-8)\leq 2$ if and only if $N\geq 8$. Since $u^\ast \in H_0^1(\Omega)$ for convex smooth domains then the previous result has no interest for dimensions $N\geq 8$ and we prefer to state it in dimensions $5\leq N\leq 7$.
\end{remark}

\begin{corollary}\label{L3}

Let $N\geq 7$, $f$ be a function satisfying (\ref{convexa}) and $\Omega\subset \mathbb{R}^N$ be a smooth bounded domain. Let $u^\ast$ be the extremal
solution of ($P_\lambda$). Suppose that $u^\ast \in L^3(\Omega)$. Then $u^\ast \in H_0^1(\Omega)$.

\end{corollary}

\begin{proof} Applying Proposition \ref{fijate} we deduce that $f(u^\ast)\in L^\frac{3}{2}(\Omega)$. By H{\"o}lder inequality we obtain $u^\ast f(u^\ast)\in L^1(\Omega)$. It follows, for every $\lambda\in (0,\lambda^\ast)$, that

$$ \int_\Omega \vert \nabla u_\lambda \vert^2=\lambda\int_\Omega u_\lambda f(u_\lambda)\leq\lambda^\ast\int_\Omega u^\ast f(u^\ast).$$

Hence $\Vert u_\lambda\Vert_{H_0^1(\Omega)}$ is uniformly bounded in $\lambda\in (0,\lambda^\ast)$. We conclude that $u^\ast \in H_0^1(\Omega)$.
\end{proof}

\end{document}